\documentclass[11pt,reqno]{amsproc}

\usepackage{mathrsfs}
\usepackage{amssymb,latexsym}
\usepackage[osf]{mathpazo}

\usepackage{amsmath,amsthm,amsfonts,amssymb,amscd,flafter,pinlabel}
\usepackage{mathtools}
\usepackage[mathscr]{eucal}
\usepackage{graphics}
 \usepackage[all]{xy}
\usepackage{caption, subcaption}
\usepackage[usenames,dvipsnames]{color}
\usepackage{graphicx}

\usepackage{pgf}
\usepackage{tikz}
\usetikzlibrary{arrows,automata}
\usepackage[latin1]{inputenc}

\usepackage[colorlinks=true, urlcolor=NavyBlue, linkcolor=NavyBlue, citecolor=NavyBlue, pdftitle={Computing the knot Floer transverse link invariant}, pdfauthor={Peter Lambert-Cole, David Shea Vela-Vick}, pdfsubject={Knot Floer Homology}, pdfkeywords={contact structures, transverse knots, 57M27, 57R58, 57R17}]{hyperref}

\newtheorem{theorem}{Theorem}[section]

\newtheorem{proposition}[theorem]{Proposition}

\newtheorem{lemma}[theorem]{Lemma}

\theoremstyle{definition}

\theoremstyle{definition}

\newcommand\HFKh{\widehat{HFK}}
\newcommand\HFKt{\widetilde{HFK}}

\newcommand\CFKh{\widehat{CFK}}

\newcommand\CFKt{\widetilde{\mathrm{CFK}}}

\newcommand\alphas{\mbox{\boldmath$\alpha$}}
\newcommand\betas{\mbox{\boldmath$\beta$}}
\newcommand\gammas{\mbox{\boldmath$\gamma$}}
\newcommand\Sym{{\rm {Sym}}}
\newcommand\ws{\mathbf w}
\newcommand\zs{\mathbf z}
\newcommand\FF{\mathbb F}
\newcommand\M{\mathcal M}

\newcommand\TT{\mathbb{T}}

\newcommand\SH{\mathcal{H}}

\def\x{\mathbf{x}}

\def\y{\mathbf{y}}
\def\z{\mathbf{z}}
\def\w{\mathbf{w}}

\newcommand{\cM}{\mathcal{M}}
\newcommand{\as}{\mathbf{a}}
\newcommand{\bbs}{\mathbf{b}}

\newcommand{\RR}{\mathbb{R}}

\newcommand{\ZZ}{\mathbb{Z}}

\newcommand{\cN}{\mathcal{N}}

\begin{document}

\title[Transverse braids and combinatorial knot Floer homology]{Transverse braids and combinatorial knot Floer homology}

\author[P. Lambert-Cole]{Peter Lambert-Cole}
\address{Department of Mathematics \\ Indiana University}
\email{pblamber@indiana.edu}
\urladdr{\href{https://www.pages.iu.edu/~pblamber/}{https://www.pages.iu.edu/\~{}pblamber}}

\author[D.S. Vela-Vick]{David Shea Vela-Vick}
\address{Department of Mathematics \\ Louisiana State University}
\email{shea@math.lsu.edu}
\urladdr{\href{https://www.math.lsu.edu/~shea/}{https://www.math.lsu.edu/\~{}shea}}

\maketitle


\begin{abstract}
We describe a new method for combinatorially computing the transverse invariant in knot Floer homology. Previous work of the authors and Stone used braid diagrams to combinatorially compute knot Floer homology of braid closures.  However, that approach was unable to explicitly identify the invariant of transverse links that naturally appears in braid diagrams.  In this paper, we improve the previous approach in order to compute the transverse invariant.  We define a new combinatorial complex that computes knot Floer homology and identify the BRAID invariant of transverse knots and links in the homology of this complex.
\end{abstract}

\section{Introduction} 
\label{sec:introduction}

Every link $L$ transverse to the standard contact structure on $S^3$ is the closure of braid $\rho$ that is unique up to conjugation and positive stabilization \cite{Be,Wr}.  For any braid $\rho \in B_n$, there is a natural multi-pointed Heegaard diagram $\SH_\rho$ associated to the corresponding braid closure. This braid diagram is a classic Heegaard decomposition of the link complement, modified to encode the braiding.  The diagram $\SH_{\rho}$ determines a bigraded complex $\CFKh(\SH_{\rho})$ whose homology is the knot Floer homology of the mirror of the braid closure.  In addition, the diagram determines a closed generator $\x[\SH_{\rho}]$ whose homology class in $\HFKh(\SH_{\rho})$ is an invariant of the transverse link determined by $\rho$.  This class $t(L)$ is the BRAID invariant of the transverse link $L$, introduced by Baldwin,  V\'ertesi, and the second author \cite{BVV}. 

BRAID is equivalent to two other powerful transverse link invariants arising in knot Floer homology, GRID and LOSS \cite{BVV}.   Ozsv\'ath, Szab\'o, and Thurston introduced the GRID invariant, which takes values in the grid version of knot Floer homology \cite{OSzT} and is easily computed for knots with small grid number. Lisca, Ozsv\'ath, Stipsicz, and Szab\'o then used open book decompositions to define an invariant, commonly referred to as LOSS, of (null-homologous) Legendrian and transverse knots in arbitrary 3--manifolds \cite{LOSS}.  These invariants have been successfully applied to distinguish and classify transverse representatives of various knot types.  

The BRAID invariant possesses advantages over its predecessors. Most notably, it is a manifestly transverse invariant as it is defined in terms of transverse knots and links that are braided about open books.  When the grid size, or equivalently arc index, exceeds the high teens, GRID cannot be efficiently computed.  As the braid complexity of a transverse link is relatively independent of arc index, BRAID potentially expands our ability to effectively distinguish transverse knots.

Together with Stone \cite{LSV}, the authors used this construction to give a new, combinatorial method for computing knot Floer homology for knots and links in $S^3$.  By a sequence of stabilizations and isotopies, the diagram $\SH_{\rho}$ is modified to $\SH_{nice}$.  This diagram is nice in the sense of Sarkar and Wang \cite{SW} and thus the differential on the complex $\CFKt(\SH_{nice})$ can be computed explicitly.  Pseudo-holomorphic curve techniques give a chain homotopy equivalence $F: \CFKt(\SH_{\rho}) \rightarrow \CFKt(\SH_{nice})$.  Thus, the bigraded ranks of $\HFKt$ of the braid closure can be computed combinatorially.  However, the chain map $F$ itself, and in particular the image of the BRAID invariant, could not be computed explicitly. 

In this paper, we describe an algebraic method to compute BRAID and prove the following theorem.

\begin{theorem}
Let $\rho \in B_n$ be a braid with transverse link closure $L$.  There is an associated complex $C(\SH_{\rho})$ that is combinatorially computable and an isomorphism on homology
\[
	(GF)_*: \HFKt(\SH_{\rho}) \rightarrow H_*(C(\SH_{\rho}))
\]
This complex supports a canonical generator $\x[\SH_{\rho}]$ such that
\[
	(GF)_*(t(L)) = \left[ \x[\SH_{\rho}] \right]
\]
\end{theorem}

Specifically, we show how to identify the transverse invariant from the complex $\CFKt(\SH_{nice})$ without computing the chain map $F$.  Starting with the pair of multi-pointed Heegaard diagrams $\SH_\rho$ and $\SH_{nice}$, we define a new chain complex $C(\SH_{\rho})$.  This new chain complex is homotopy equivalent to $\CFKt(\SH_{\rho})$ and possesses the same underlying module as $\CFKt(\SH_{\rho})$.   The differential is defined analogously to that of $\CFKt(\SH_{\rho})$ by assigning counts to Whitney disks via the rule
\[
	d(\x) \coloneqq \sum_{\y \in \TT_{\alpha} \cap \TT_{\beta}} \sum_{ \substack{ \phi \in \pi_2(\x,\y) \\ \mu(\phi) = 1 \\ n_z(\phi) = n_w (\phi) = 0}} \cN(\phi) \cdot \y.
\]
However, the count $\cN(\phi)$ of representatives of a domain $\phi$ is determined combinatorially by the differential on $\CFKt(\SH_{nice})$, instead of geometrically by counting pseudo-holomorphic representatives. The differential on $\CFKt(\SH_{nice})$ determines a chain-homotopy equivalence $G: \CFKt(\SH_{nice}) \to C(\SH_\rho)$ as well.  The resulting composition $GF$ can similarly be interpreted as a standard triangle map defined by assigning modified counts to Whitney triangles.  Importantly, while we cannot identify the image of $\x[\SH_{\rho}]$ under $F$ itself, in Proposition~\ref{prop:GF-transverse} we are able to determine its image under the composition $GF$, showing
\begin{equation}
GF(\x[\SH_{\rho}]) = \x[\SH_{\rho}]
\end{equation}
As a consequence, the transverse invariant $t(L)$ is combinatorially computable directly from the braid.

\subsection*{Acknowledgements} 
\label{sub:acknowledgements}
We would like to acknowledge the National Science Foundation and Louisiana State University for sponsoring the 2012 LSU Research Experience for Undergraduates (NSF Grant DMS-1156663).  The mathematics presented here originated as an offshoot of this program.  Vela-Vick would also like to acknowledge partial support from NSF Grant DMS-1249708.

\section{Preliminaries} 
\label{sec:preliminaries}
In what follows, we assume familiarity with knot and braid theory, as well as elementary aspects of contact geometry and Legendrian and transverse links. The interested reader is encouraged to consult Birman's book \cite{Bi} and Etnyre's notes \cite{Et} for comprehensive introductions to braid theory and to Legendrian and transverse knot theory.


\subsection{Knot Floer homology} 
\label{sub:hfk}
We begin by summarizing some basic definitions and results concerning knot Floer homology. We refer the reader to the papers by Ozsv\'ath and Szab\'o \cite{OS3} and Rasmusen \cite{Ra} for a more in-depth discussion of this material.   Throughout this manuscript, we work with $\FF = \ZZ/2\ZZ$--coefficients. 

Recall that to each oriented link $K$ in the 3--sphere, one can associate a {\it multi-pointed Heegaard diagram} $\SH = (\Sigma,\alphas,\betas,\zs,\ws)$.  In this case, the triple $(\Sigma,\alphas,\betas)$ specifies a Heegaard diagram for $S^3$ and the link $K$ is obtained from the basepoints $\zs = \{z_i\}$ and $\ws = \{w_i\}$ by connecting the $z$ to $w$-basepoints and $w$ to $z$-basepoints by properly embedded arcs in the $\alpha$ and $\beta$-handlebodies respectively which avoids the compression disks specified by the curves in $\alphas$ and $\betas$.  More generally, a {\it multi-pointed Heegaard triple} $\SH = (\Sigma,\alphas,\betas,\gammas,\zs,\ws)$ is a collection of three sets of curves $\alphas,\betas,\gammas$ such that each pair $(\alphas,\betas),(\alphas,\gammas),$ and $(\betas,\gammas)$ determines multi-pointed Heegaard diagrams.  

The collections $\alphas$ and $\betas$ specify tori $\TT_\alpha$ and $\TT_\beta$ in $\Sym^{g+n-1}(\Sigma)$. The complex $\CFKt(\SH)$ is the $\FF$-vector space freely generated by the intersections in $\TT_\alpha \cap \TT_\beta$. For each Whitney disk $\phi \in \pi_2(\x,\y)$, we let $n_{z_i}(\phi)$ and $n_{w_i}(\phi)$ denote the local multiplicity of $\phi$ at $z_i$ and $w_i$ respectively.  We denote by $n_{\zs}(\phi)$ and $n_{\ws}(\phi)$ the sums of the local multiplicities at all of the $z$ and $w$-basepoints respectively. The chain group can be endowed with two absolute gradings, the Maslov (homological) grading $M(\x)$ and Alexander grading $A(\x)$, which are determined up to an overall shift by the formulas
\[
	M(\x) - M(\y) = \mu(\phi) - 2n_{\ws}(\phi) \;\;\; \text{and} \;\;\; A(\x) - A(\y) = n_{\zs}(\phi) - n_{\ws}(\phi),
\]
where $\phi \in \pi_2(\x,\y)$ and $\mu(\phi)$ denote the Maslov index of the Whitney disk $\phi$.  The differential on the complex $\CFKt(\SH)$ is defined by
\[
	\widetilde{\partial}(\x) = \sum_{\y \in \TT_\alpha \cap \TT_\beta} \sum_{\substack{\phi \in \pi_2(\x,\y),\\ \mu(\phi) = 1, \\ n_{\zs}(\phi) = n_{\ws}(\phi) = 0}} \# \widehat{\M}(\phi) \cdot \y.
\]
The {\it tilde} version of knot Floer homology is then
\[
	\HFKt(K) := H_*(\CFKt(\SH),\widetilde{\partial}),
\]
and is an invariant of the link $K$ and the number of $z$ or $w$-basepoints. Its relation to the {\it hat} version of knot Floer homology is given by
\[
	\HFKt(K) = \HFKh(K) \otimes V^{\otimes n},
\]
where $V$ is a 2--dimensional vector space supported in bi-gradings $(0,0)$ and $(-1,-1)$.

Let $\Pi_{\alpha,\beta}$ denote the group of {\it periodic domains} in the Heegaard diagram $\SH = (\Sigma,\alphas,\betas,\zs,\ws)$.  Recall that a 2-chain is periodic if its boundary is the union of some number of $\alpha$ and $\beta$ curves.  Let $\Pi^0_{\alpha,\beta}$ denote the subgroup of periodic domains that avoid $\ws \cup \zs$.  As a group, $\Pi^0_{\alpha,\beta}$ is isomorphic to $H^2(S^3 \setminus K; \ZZ)$.  The Heegaard diagram $\SH$ is {\it admissible} if every domain in $\Pi^0_{\alpha,\beta}$ has both positive and negative multiplicities.  The groups $\Pi_{\alpha,\beta,\gamma}, \Pi^0_{\alpha,\beta,\gamma}$ of periodic domains and admissibility are defined similarly for a triple $\alpha,\beta,\gamma$.

Finally, given an admissible Heegaard triple $\SH = (\Sigma,\alphas,\betas,\gammas,\zs,\ws)$, there is an induced chain map $F: \CFKt(\SH_{\alpha,\beta}) \otimes \CFKt(\SH_{\beta,\gamma}) \rightarrow \CFKt(\SH_{\alpha,\gamma})$ defined as
\[F(\x_1 \otimes \x_2) \coloneqq \sum_{\y \in \TT_{\alpha} \cap \TT_{\beta}} \sum_{\substack{\psi \in \pi_2(\x_1,\x_2,\y) \\ \mu(\psi) = 0 \\ n_z(\psi) = n_w(\psi) = 0}} \# \widehat{\cM}(\psi) \cdot \y\]


\subsection{Combinatorial computations} 
\label{sub:combinatorial_computations}

In \cite{LSV}, Stone and the authors described an algorithm for combinatorially computing knot Floer homology. The algorithm begins with a braid presentation of a given link $K$ and produces an explicit, nice multi-pointed Heegaard diagram in the sense of Sarkar and Wang \cite{SW}. We outline below the construction from \cite{LSV}.

A multi-pointed Heegaard diagram $\SH$ is {\it nice} if every region in $\Sigma \backslash (\alphas \cup \betas)$ which does not containing a basepoint is topologically a disk with at most four corners.  In other words, every region in the complement of the $\alpha$ and $\beta$-curves either contains a $z$-basepoint, or is a bigon or square. If a multi-pointed Heegaard diagram $\SH$ is nice, the differential on $\CFKt(\SH)$ can be computed combinatorially by counting embedded, empty rectangles and bigons connecting generators \cite{SW}.  

There is a well-known isomorphism between the $n$-stranded braid group $B_n$ and the mapping class group $MCG(D,n)$ of the disk with $n$ marked points. Let $(D,n)$ denote the unit disk in $\RR^2$ with $n$ evenly spaced marked points $\z = \{z_1,\dots,z_n\}$ along the horizontal axis. The isomorphism identifies the $i^{th}$ standard generator $\sigma_i$ of the Artin braid group with the positive half-twist about the horizontal arc joining the $i^{th}$ and $(i+1)^{st}$ marked points on $D$.

\begin{figure}[htpb]
\labellist
	\hair 2pt
	\small
	\pinlabel $\gamma$ at 40 70
	\pinlabel $\gamma$ at 185 70
\endlabellist
\begin{center}
\includegraphics[scale=1.0]{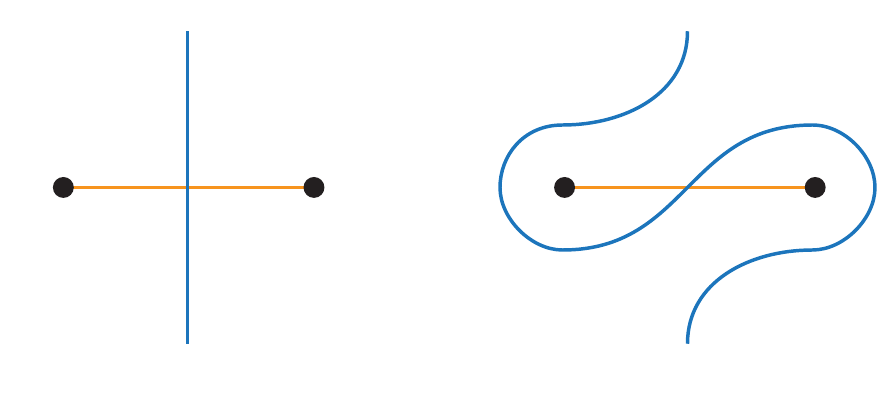}
\caption{A Dehn half-twist $\tau_\gamma$ about the arc $\gamma$.}
\label{fig:half_twist}
\end{center}
\end{figure}

Let $\as = \{a_1,\dots,a_{n-1}\}$ denote an arc-basis for $(D,\z)$ consisting of $n-1$ vertical arcs, such that each component of $D \backslash \as$ contains exactly one of the $z$-basepoints. Next, let $\bbs$ denote a second arc-basis which is obtained from $\as$ by applying small isotopies which shift the endpoints of the $a_i$ along the orientation of $\partial D$ and results in a single transverse intersection $x_i = a_i \cap b_i$.

Take a second copy $D'$ of the disk $D$ with identical basepoints $\w = \{w_1,\dots,w_n\}$. For a homeomorphism $\rho: (D,\z) \to (D,\z)$, we endow $D'$ with two arc bases: the first $\as'$ is identical to $\as$, while the second $\bbs'$ is obtained from $\bbs$ by setting $b'_i \coloneqq \rho(b_i)$.  By perturbing, if necessary, we can assume that the $a'_i$ and $b'_j$ meet transversally in a single point.  We then obtain an admissible, multi-pointed Heegaard diagram $\SH_\rho = (\Sigma ,\betas_{\rho},\alphas_{\rho},\zs,\ws)$ by setting $\Sigma = D \cup -D'$, $\alpha_i = a_i \cup a'_i$ and $\beta_i = b_i \cup b'_i$. 

Note that in the definition of $\SH_{\rho}$, we have interchanged the roles of the $\alpha$ and $\beta$-curves. Topologically, this has the effect of reversing the orientation of the ambient manifold --- in this case, the orientation of $S^3$. On the level of knot Floer homology groups, we have
\[
	\HFKt(\SH_{\rho}) \cong \HFKt(-S^3,K) \cong \HFKt(S^3,m(K)).
\]

We call a homeomorphism $\rho$ {\it efficient} if it minimizes intersections amongst the $a'_i$ and $b'_j$ within its mapping class. Such maps give rise to multi-pointed Heegaard diagrams which we also call {\it efficient}, and which are very close to being nice: they contain at most $n-1$ bad regions, each with six sides.  To convert an efficient, multi-pointed Heegaard diagram into a nice one, we apply a ``stabilization trick'' that was first described in \cite{Hales}. This trick consists of two steps:
\begin{enumerate}
	\item[(1)] For each 6-sided bad region $R$ in $\SH_\rho$, stabilize $\SH_\rho$ as in Figure~\ref{fig:stabilizedA} by attaching a 1--handle to $\Sigma$ with one foot in $R$ and another in a region containing a $z$-basepoint.
	\item[(2)] Isotope the new $\beta$-curves as in Figure~\ref{fig:stabilizedB} by applying finger moves across the $\alpha$-edges until reaching regions containing basepoints.
\end{enumerate}
The resulting diagram $\SH_{nice}$ after stabilizing and isotoping is nice \cite[Proposition 3.3]{LSV}.

\begin{figure}[htpb]
\begin{center}
\begin{subfigure}{.45\textwidth}
	\centering
	\labellist
		\small\hair 2pt
		\pinlabel $\widehat{\alpha_i}$ at 51 103
		\pinlabel $\widehat{\beta_i}$ at 68 87
		\pinlabel $z_i$ at 25 150
	 	\pinlabel $w_i$ at 8 81
		\pinlabel $\alpha_i$ at -7 25
	 	\pinlabel $\alpha_{i+1}$ at 85 25
		\pinlabel $\widehat{t_i}$ at 43 75
	\endlabellist
	\includegraphics[scale=1.0]{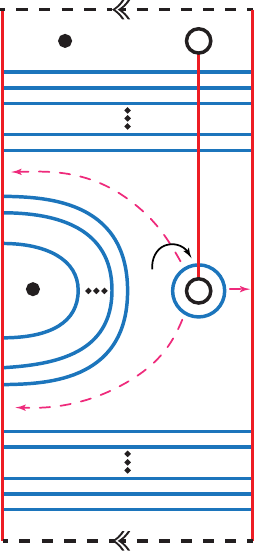}
	\caption{The multi-pointed Heegaard diagram after stabilizing. Finger moves are applied as indicated by the (pink) dashed lines.}
	\label{fig:stabilizedA}
\end{subfigure}
\hspace{0.3cm}
\begin{subfigure}{.45\textwidth}
	\centering
	\labellist
		\small\hair 2pt
		\pinlabel $\widehat{\alpha_i}$ at 80 76
		\pinlabel $\widehat{\beta_i}$ at 60 90
		\pinlabel $z_j$ at 25 150
	 	\pinlabel $w_k$ at 8 81
		\pinlabel $\beta_j$ at -7 25
	 	\pinlabel $\beta_{j+1}$ at 85 25
		\pinlabel $\widehat{t_i}$ at 57 60
	\endlabellist
	\includegraphics[scale=1.0]{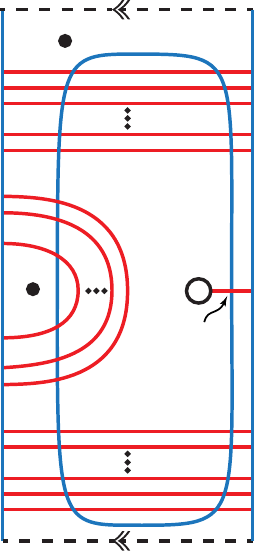}
	\caption{The multi-pointed Heegaard diagram $\SH_{nice}$, viewed from the ``dual'' perspective with $\beta$-curves appear visually straightened.}
	\label{fig:stabilizedB}
\end{subfigure}
\caption{The Stabilization Trick}
\label{fig:stabilization}
\end{center}
\end{figure}
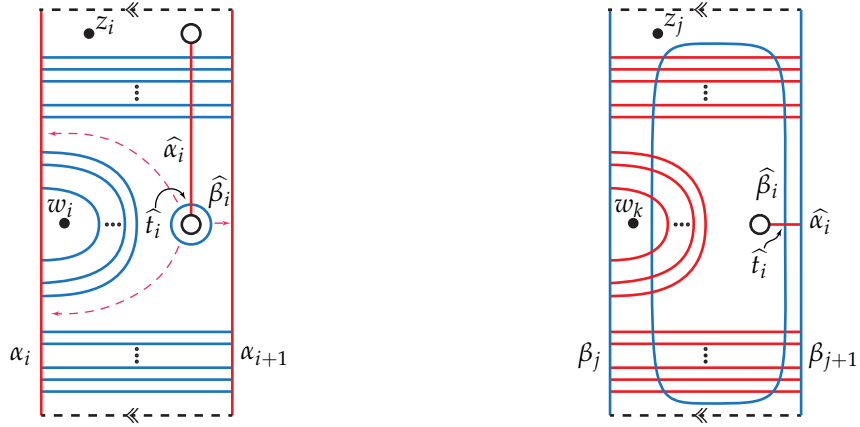


\subsection{The transverse invariant} 
\label{sub:transverse}
Let $L$ be a transverse link in $(S^3,\xi_{std})$ which is braided about the standard (disk) open book decomposition for $(S^3,\xi_{std})$.  To this braid, we associate a multi-pointed Heegaard diagram $\SH_\rho$, as above.  We call a diagram obtained in this way a {\it braid diagram}. 

In the context of Heegaard Floer theory, braid diagrams first appeared in the work of Baldwin, V\'ertesi and the second author \cite{BVV}, and were used to establish an equivalence of transverse invariants in knot Floer homology.  Observe that the diagram $\SH_\rho$ supports a distinguished generator $\x[\SH_{\rho}] = \{t_1,\dots,t_{n-1}\}$, which is the union of the unique intersections between the $\alpha$ and $\beta$-curves contained on the disk $D \subset \Sigma$. It was shown in \cite{BVV} that the class $[\x[\SH_{\rho}]] \in \HFKh(-S^3,L)$ is an invariant of the transverse link $L$. This invariant is denoted $t(L)$ and is commonly referred to as the {\it BRAID invariant} of transverse knots.  The Maslov and Alexander gradings of $t(L)$ are given by
\[
	M(t(L)) = sl(L) + 1 \;\;\;\;\; \text{and} \;\;\;\;\; A(t(L)) = \frac{sl(L)+1}{2},
\]
where $sl(L)$ is the self-linking number of $L$.  The self-linking number of the closure of an $n$-braid $\rho$ is
\[sl(L) = wr(\rho) - n\]
where $wr(\rho)$ is the writhe.

In this paper, we work with the tilde version of knot Floer homology $\HFKt(-S^3,L)$ instead of $\HFKh(-S^3,L)$.  However, the generator $\x[\SH_{\rho}]$ determines a class in $\HFKt(-S^3,L)$ as well.  Moreover, there exists a canonical projection map $p: \HFKt(-S^3,L) \rightarrow \HFKh(-S^3,L)$ with corresponding section $s: \HFKh(-S^3,L) \rightarrow \HFKt(-S^3,L)$, such that $s(\x[\SH_{\rho}]) =  \x[\SH_{\rho}]$.


\section{Identifying the transverse invariant} 
\label{sec:identify_invt}

Throughout this section, we assume that a transverse link $L$ is given as the closure of a braid $\rho$.  By abuse of notation, we let $\rho$ also denote an efficient homeomorphism in the corresponding mapping class.

\subsection{The complexes $C_{i,j}$} 
\label{sub:the_complexes_c__i_j}

In this subsection, we inductively define a sequence of complexes $C_{i,j}$ and a sequence of chain-homotopy equivalences $F_{i,j}: \CFKt(\SH_{\rho}) \rightarrow C_{i,j}$.  The differentials on the complexes are defined analogously to those of $\CFKt$, except that the counts associated to each Whitney disk are determined combinatorially.  Similarly, the chain maps are defined analogously to standard triangle maps, except with modified counts for Whitney triangles.

Let $\SH_{\rho}$ be an efficient, multi-pointed Heegaard diagram for a braid closure.  Via the ``stabilization trick'', there is a sequence of multipointed Heegaard diagrams
\[
	\SH_{\rho},\SH_s,\SH_0,\SH_1,\dots,\SH_N = \SH_{nice}
\]
where
\begin{enumerate}
\item $\SH_s = (\Sigma_g,\betas_{\rho} \cup \widehat{\betas}, \alphas_{\rho} \cup \widehat{\alphas}, \zs,\ws)$ is obtained from $\SH_{\rho}$ by $g$ simultaneous stabilizations in neighborhoods of the appropriate $z$--basepoints on the disk $D$;
\item $\SH_0 = (\Sigma_g,\betas_0,\alphas_0,\zs,\ws)$ is obtained from $\SH_s$ by handlesliding across the $\widehat{\betas}$ curves so that each $\widehat{\alpha}_k$ intersects the original 2-sphere along an arc from the region containing $z_k$ to the unique hexagon in annulus bounded by $\alpha_{k-1}$ and $\alpha_k$; and
\item $\SH_{i} = (\Sigma_g,\betas_{i},\alphas_{0},\zs,\ws)$ is obtained from $\SH_{i-1}$ by an elementary isotopy of some $\widehat{\beta}_k$ across some $\alpha$ curve.  In addition, by a small Hamiltonian isotopy we can assume that each curve of $\betas_i$ intersects its corresponding curve in $\betas_0,\dots,\betas_{i-1}$ transversely in two points.
\end{enumerate}

Finally, let $\SH_{0'} = (\Sigma_g,\betas_{0'},\alphas,\zs,\ws)$ be a multi-pointed Heegaard diagram where the curves of $\betas_{0'}$ are small Hamiltonian isotopes of the curves of $\betas_0$ which each intersect their counterpart in $\betas_0,\dots,\betas_N$ transversely in two points.

\begin{figure}[htpb]
\labellist
	\hair 2pt
	\small
	\pinlabel $q$ at 252 30
	\pinlabel $b$ at 252 10
	\pinlabel $y'_i$ at 294 28
	\pinlabel $y_i$ at 213 27
	\pinlabel $\alpha_j$ at -6 18
	\pinlabel $\widehat{\beta}_k$ at -6 47
\endlabellist
\begin{center}
\includegraphics[scale=1.0]{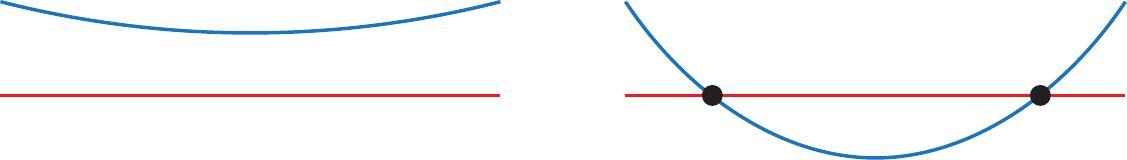}
\caption{An elementary isotopy introduces two intersection points to $\alphas \cap \betas$.}
\label{fig:bigon}
\end{center}
\end{figure}
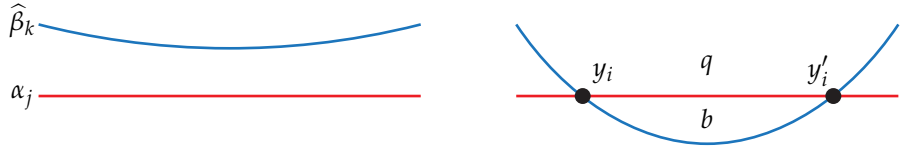

Each isotopy move from $\SH_0$ to $\SH_{nice}$ introduces a pair of intersection points $y_i,y'_i$ to $\alphas \cap \betas$ as in Figure~\ref{fig:bigon}.  Let $D_i$ be the submodule spanned by generators with a vertex at $y_i$ or $y'_i$ and no vertices at any $y_j$ or $y'_j$ for $j > i$.  Then, for each $0 \leq i \leq N$, there is a decomposition of modules
 \[
 	\CFKt(\SH_i) = \CFKt(\SH_0) \oplus D_1 \oplus \cdots \oplus D_i.
\]
In particular, when $i = N$, we have $\CFKt(\SH_N) = \CFKt(\SH_{nice})$.  We additionally see that the modules $D_i$ have direct sum decompositions
\[ 
	D_i  = D_{i,1} \oplus \cdots \oplus D_{i,n_i}
\]
where $D_{i,j}$ is spanned by two generators $\x_{i,j},\x'_{i,j}$ satisfying $\x'_{i,j} = \left( \x_{i,j} \setminus \{y_i\} \right) \cup \{y'_i\}$.  Define an increasing sequence of submodules
\[
	C_{i,j} \coloneqq \CFKt(\SH_0) \oplus \bigoplus_{k = 1}^{i-1} D_k \oplus \bigoplus_{m = 1}^{j} D_{i,m}
\]
for $i = 1,\dots,N$ and $j = 1,\dots,n_i$.  Endow each intermediate module $C_{i,j}$ with the map
\[ 
	d_{i,j} (\x) \coloneqq \sum_{\y} \sum_{ \substack{ \phi \in \pi_2(\x,\y) \\ \mu(\phi) = 1 \\ n_z(\phi) = n_w(\phi) = 0}} \cN_{i,j}(\phi) \cdot \y
\]
where $\cN(\phi)$ is a modified count of representatives of the Whitney disk $\phi$ defined inductively as follows.  

First, we set $d_{N,n_N}$ on $C_{N,{n_N}} = \CFKt(\SH_{nice})$ to be exactly the differential on $\CFKt(\SH_{nice})$.  Thus, define
\[
	\cN_{N,n_N}(\phi) \coloneqq \# \widehat{\cM}(\phi).
\]
where $\cM(\phi)$ denotes the moduli of pseudo-holomorphic representatives.  
For each pair of intersection points $y_i,y'_i$ induced by the finger moves, there is a unique bigon 2--chain $B_i$ with corners at $y_i$ and $y'_i$.  This bigon determines a Whitney disk $B_{i,j} \in \pi_2(\x_{i,j},\x'_{i,j})$.  Inductively define the count $\cN_{i,j}$ by setting
\[ \cN_{i,j}(\phi) \coloneqq \cN_{i,j+1}(\phi) + \sum_{ \substack{\phi_1 \in \pi_2(\x_1, \y'_{i,j}) \\ \phi_2 \in \pi_2(\y_{i,j},\x_2) \\ \mu(\phi_1) = \mu(\phi_2) = 1 \\ \phi_1 - B_{i,j} + \phi_2 = \phi \\ n_z(\phi_1) = n_w(\phi_1) = 0 \\ n_z(\phi_2) = n_w(\phi_2) = 0}} \cN_{i,j+1}(\phi_1) \cdot \cN_{i,j+1}(\phi_2) \]
In defining the counts, we use the convention that $(i,n_i) \equiv (i+1,0)$.

Next, we define chain-homotopy equivalences between $\CFKt(\SH_0)$ and each $C_{i,j}$ by a modified count of Whitney triangles.  The three sets of curves $\betas_i,\betas,\alphas$, along with the basepoints, determine a Heegaard triple.  There is a unique generator $\Theta_i \in \TT_{\beta_0} \cap \TT_{\beta_i}$ of maximal Maslov grading, which is closed.  Define maps $F_{i,j}: \CFKt(\SH_0) \rightarrow C_{i,j}$ by the rule
\[
	F_{i,j}(\x) \coloneqq \sum_{\y \in \TT_{\alpha} \cap \TT_{\beta_i}} \sum_{\substack{ \psi \in \pi_2(\Theta_i,\x,\y) \\ \mu(\psi) = 0 \\ n_z(\psi) = n_w(\psi) = 0}} \cN_{i,j}(\psi) \cdot \y
\]
The counts $\cN_{i,j}$ are again defined inductively.  We set the map $F_{N,n_N}$ to be the induced chain homotopy equivalence from $\CFKt(\SH_0)$ to $\CFKt(\SH_{nice})$.  Thus, define
\[
	\cN_{N,n_N}(\psi) \coloneqq \# \widehat{\cM}(\psi)
	\]
We then define the triangle counts inductively using the formula
\[
	\cN_{i,j}(\psi) \coloneqq \cN_{i,j+1}(\psi) + \sum_{\substack{ \psi_1 \in \pi_2(\Theta,\x,\x'_{i,j}) \\ \phi_2 \in \pi_2(\x_{i,j},\y) \\ \mu(\phi_1) = 1; \mu(\psi_2) = 0 \\ \psi_1 - B_{i,j} + \phi_2 = \psi \\ n_z(\phi_1) = n_w(\phi_1) = 0 \\ n_z(\psi_2) = n_w(\psi_2) = 0}} \cN_{i,j+1}(\phi_1) \cdot \cN_{i,j+1}(\psi_2)
\]

We now state the main result of this subsection.

\begin{proposition}
\label{prop:combinatorial-complexes}
For any $i = 1,\dots,N$ and $j = 1,\dots,n_i$,
\begin{enumerate}
\item $(C_{i,j},d_{i,j})$ is a chain complex,
\item $F_{i,j}$ is a chain map, and
\item $F_{i,j}: \CFKt(\SH_0) \rightarrow C_{i,j}$ is a chain-homotopy equivalence.
\end{enumerate}
\end{proposition}

This proposition follows easily from the following two lemmas.  First, we use a standard fact in homological algebra to contract differentials and simplify the complex.

\begin{lemma}[Cancellation Lemma]
\label{lemma:null-homotopy}
Let $(C \oplus D, d)$ be a chain complex with differential
\[d = 
\begin{bmatrix}
d_{CC} & d_{CD} \\
d_{DC} & d_{DD}
\end{bmatrix}
\]
such that $(D, d_{DD})$ is a contractible complex with null-homotopy $H: D \rightarrow D$.  Let $(C,d')$ be the complex with twisted differential 
\[
	d' = d_{CC} + d_{CD} \circ H \circ d_{DC}.
\]
Then, the maps $F: (C,d') \rightarrow (C \oplus D,d)$ and $G: (C \oplus D,d) \rightarrow (C,d')$ defined by
\[F \coloneqq \text{Id}_C \oplus H \circ d_{DC}  \qquad \qquad 
G \coloneqq \text{Id}_C + d_{CD} \circ H
\]
are chain-homotopy equivalences.
\end{lemma}

Second, we establish the existence of a sequence of differentials to contract.  

\begin{lemma}
\label{lemma:bigon-arrow}
For any $i = 1,\dots,N$ and $j = 1,\dots,n_i$, we have $\langle d_{i,j} (\x_{i,j}), \x'_{i,j} \rangle = 1$.
\end{lemma}

We defer the proof of Lemma~\ref{lemma:bigon-arrow} until the end of this subsection.  Using Lemmas \ref{lemma:null-homotopy} and \ref{lemma:bigon-arrow}, we prove Proposition~\ref{prop:combinatorial-complexes}.

\begin{proof}[Proof of Proposition~\ref{prop:combinatorial-complexes}]
All three statements are clearly true for $i = N$ and $j = n_N$ since the differential and triangle map are defined by counting pseudo-holomorphic representatives.

To prove the proposition, suppose by induction that all three statements are true for $(i,j+1)$.  Let $d^D_{i,j+1}$ denote the restriction of $d_{i,j+1}$ to the submodule $D_{i,j+1}$.  Then we have that $d^D_{i,j+1} \x_{i,j+1} = \x'_{i,j+1}$ by Lemma~\ref{lemma:bigon-arrow}.  Consequently, $(D_{i,j+1},d^D_{i,j+1})$ is a contractible complex with null-homotopy $H_{i,j+1}$ defined by setting $H(\x'_{i,j+1}) = \x_{i,j+1}$.  Using the Cancellation Lemma (Lemma~\ref{lemma:null-homotopy}), we can contract $D_{i,j+1}$ to obtain a twisted differential on $C_{i,j}$ and a chain-homotopy equivalence $G_{i,j}: C_{i,j+1} \rightarrow C_{i,j}$.  This twisted differential is precisely $d_{i,j}$ and $F_{i,j} = G_{i,j} \circ F_{i,j+1}$.  All three statements are now clear for $(i,j)$.
\end{proof}

We now return to the proof of Lemma~\ref{lemma:bigon-arrow}.  In preparation, we show that, like the count of pseudo-holomorphic representatives, if the count $\cN_{i,j}$ associated to a Whitney disk or triangle is nonzero, then the corresponding domain in the Heegaard diagram is positive.

\begin{lemma}
\label{lemma:positive}
Let $\phi$ denote a Whitney disk in the diagram $\SH_i$, let $\psi$ denote a Whitney triangle in the multidiagram determined by $\betas_i,\betas_0,\alphas$, and let $B_i$ denote the bigon region in $\SH_i$ with corners at $y_i$ and $y'_i$.
\begin{enumerate}
\item If the count $\cN_{i,j}(\phi)$ is nonzero, then the 2-chain $D(\phi) + B_i$ is positive.
\item If the count $\cN_{i,j}(\psi)$ is nonzero, then the 2-chain $D(\psi) + B_i$ is positive.
\end{enumerate}
\end{lemma}

\begin{proof}
We begin by observing that if $D(\phi) + B_{i+1}$ is positive as a 2-chain in $\SH_{i+1}$, then $D(\phi)$ is positive as a 2-chain in $\SH_{i}$.

Also, both statements are clearly true for $(i,j) = (N,n_N)$ since both the differential and triangle maps are determined by counting holomorphic representatives.  If a holomorphic representative exists, then by positivity of intersection, the 2-chain corresponding to the Whitney disk or triangle must be positive.

Now, suppose by induction the statements are true for $(i,j+1)$.  If $\cN_{i,j}(\phi) \neq 0$, then either $\cN_{i,j+1}(\phi) \neq 0$, or there exist two domains $\phi_1,\phi_2$ such that $\phi_1 + \phi_2 - B_{i,j} = \phi$ and $\cN_{i,j+1}(\phi_1)$ and $\cN_{i,j+1}(\phi_2)$ are nonzero.

In the first case, then clearly $D(\phi) + B_i$ is positive by induction.   In the second case, we know by induction that $D(\phi_1) + B_i$ and $D(\phi_2) + B_i$ are both positive.  We will show the stronger statements that $D(\phi_1)$ and $D(\phi_2)$ are positive.  All $\alpha$ curves are adjacent to $z$-basepointed regions on both sides.  Since $n_{\z}(\phi_1)= 0$, this implies that the multiplicity of $D(\phi_1)$ can change by at most 1 across any segment of $\alphas$.  In particular, pick two points $b,q$ as in Figure~\ref{fig:bigon}.  Then $|n_b(\phi_1) - n_q(\phi_1)| \leq 1$.  Moreover, since $\phi_1$ has an outgoing corner at $y_i$, the multiplicities of $\phi_1$ must satisfy $n_b(\phi_1) \geq n_q(\phi_1)$.  However, since $D(\phi_1) + B_i$ is positive, this means that $n_q(\phi_1) \geq 0$.  Thus $n_b(\phi_1) \geq n_q(\phi_1) \geq 0$.  Consequently, the domain $D(\phi_1)$ is positive.  A corresponding argument shows that $D(\phi_2)$ is positive.  As a result, $D(\phi) + B_i = D(\phi_1) + D(\phi_2)$ is positive.

A similar inductive argument proves the statement for triangle counts.
\end{proof}

We now finish the proof of Lemma~\ref{lemma:bigon-arrow}.

\begin{proof}[Proof of Lemma~\ref{lemma:bigon-arrow}]
First, the domain $D(B_{i,j}) = B_i$ is clearly positive.  If $\phi \in \pi_2(\x_{i,j},\x'_{i,j})$ is any other Whitney disk, then $D(\phi - B_{i,j})$ is a periodic domain.  Since the diagram $\SH_i$ is admissible, this periodic domain has both positive and negative multiplicities.  Consequently, $D(\phi)$ is positive if and only if the negative component of $D(\phi - B_{i,j})$ is precisely $B_i$.  However, this is not possible.  If it was, then the boundary of the periodic domain must include the $\alpha$ and $\beta$ curves bounding $B_i$.  Yet, each $\beta$ curve introduced by stabilization is not a linear combination in $H_1(\Sigma_k)$ of the remaining $\alpha$ and $\beta$ curves.  Thus, it must show up with multiplicity $0$ in the boundary of any periodic domain.  In turn, $B_{i,j}$ is the unique Whitney disk in $\pi_2(\x_{i,j},\x'_{i,j})$ with a positive domain and, by Lemma~\ref{lemma:positive}, the only possible Whitney disk that could contribute to $\langle d_{i,j} (\x_{i,j}), \x'_{i,j} \rangle$.  Therefore, it sufficies to check that $\cN_{i,j}(B_{i,j}) = 1$.

The count satisfies $\cN_{N,n_N}(B_{i,j}) = \# \widehat{\cM}(B_{i,j}) = 1$ since embedded bigons have unique pseudo-holomorphic representatives.  Proceeding by induction, we see that $\cN_{i',j'}(B_{i,j}) = \cN_{i',j'+1}(B_{i,j})$ unless $B_{i,j}$ admits a decomposition $B_{i,j} = \phi_1 - B_{i',j'} + \phi_2$ for some Whitney disks with $\cN_{i',j'+1}(\phi_1) = \cN_{i',j'+1}(\phi_2) = 1$.  If this occurs, then by Lemma~\ref{lemma:positive}, the domains $D(\phi_1),D(\phi_2)$ are positive.  However, since $B_i$ is an embedded bigon, this implies a finger move introduces two new intersection points $y_{i'},y'_{i'}$ on the $\alpha$ boundary arc of $B_i$ between $y_i$ and $y'_i$.  But it is clear from Figure~\ref{fig:stabilizedB} that this never occurs in the isotopy.   Consequently, the count $\cN_{i',j'}(B_{i,j}) $ stays fixed at 1 for all $(i,j) \leq (i',j') \leq (N,n_N)$.  In particular, $\cN_{i,j}(B_{i,j}) = 1$.
\end{proof}


\subsection{Transverse invariant} 
\label{sub:transverse_invariant}
In this section, we identify the image of the transverse invariant $t(K)$ under the chain homotopy equivalences defined in the previous section.

Recall that $\x[\SH_{\rho}]$ is the generator of $\CFKt(\SH_{\rho})$ comprised of the distinguished intersections  $t_1,\dots,t_{n-1}$ that lie on the portion of the original Heegaard surface coming from $D$.  There are corresponding generators $\x[\SH_s]$ in $\CFKt(\SH_s)$ and $\x[\SH_0]$ in $\CFKt(\SH_0)$ specified by the distinguished intersections along with the unique intersection points $\widehat{t_k} \in \widehat{\alpha_k} \cap \widehat{\beta_k}$.  Finally, for $i = 1,\dots,N$, let $T(\SH_i) \subset \CFKt(\SH_i)$ denote the subspace spanned by generators whose vertices along the curves $\alpha_1,\dots,\alpha_{n-1}$ are precisely $t_1,\dots,t_{n-1}$. 

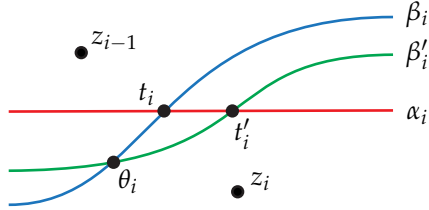
\begin{figure}[htpb]
\centering
\labellist
	\small\hair 2pt
	\pinlabel $t_i$ at 52 43
	\pinlabel $t'_i$ at 88 28
	\pinlabel $\theta_i$ at 45 10
	\pinlabel $\alpha_i$ at 154 35
	\pinlabel $\beta'_i$ at 154 58
	\pinlabel $\beta_i$ at 154 73
	\pinlabel $z_{i-1}$ at 40 63
	\pinlabel $z_i$ at 95 10
\endlabellist
\includegraphics[scale=1.0]{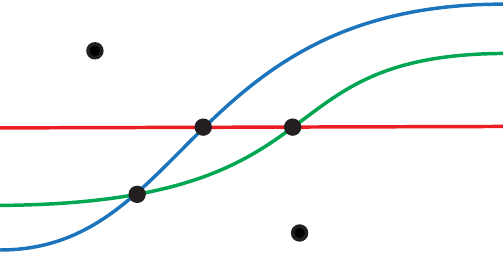}
\caption{All triangles with $n_z(\psi) = 0$ and corners at $t_i$ and $\theta_i$ must have a third corner at $t'_i$.}
\label{fig:triangle}
\end{figure}

\begin{lemma}
\label{lemma:transverse-subcomplex}
Let $\x[\SH_{\rho}],\x[\SH_s],$ and $\x[\SH_0]$ be the distinguished generators in $\SH_{\rho},\SH_s$, and $\SH_0$, respectively, and let $T(\SH_i)$ be the subspace spanned by generators containing the distinguished intersections $\{t_i \in \alpha_i \cap \beta_i\}$.  Then
\begin{enumerate}
\item the chain homotopy equivalences $f_s: \CFKt(\SH_{\rho}) \rightarrow \CFKt(\SH_s)$ and $f_h: \CFKt(\SH_s) \rightarrow \CFKt(\SH_0)$ satisfy
\[
	f_s(\x[\SH_{\rho}]) = \x[\SH_s] \qquad f_h(\x[\SH_s]) = \x[\SH_0];
\]
\item for all $i = 0,\dots,N$, the subspace $T(\SH_i)$ is a subcomplex of $\CFKt(\SH_i)$;
\item the chain homotopy equivalence $f: \CFKt(\SH_0) \rightarrow \CFKt(\SH_N)$ satisfies
\[
	f(\x[\SH_0]) \subset T(\SH_{N});
\]
\item for all $i,j$, the subspace $T(\SH_{i+1}) \cap C_{i,j}$ is a subcomplex of $C_{i,j}$; and
\item for all $i,j$, the chain homotopy equivalence $G_{i,j}: C_{i,j+1} \rightarrow C_{i,j}$ satisfies
\[G_{i,j}( T(\SH_{i+1}) \cap C_{i,j+1}) \subset T(\SH_{i+1}) \cap C_{i,j}\]
\end{enumerate}
\end{lemma}

\begin{proof}
Statements (2) and (4) follow from the same argument that $\x[\SH_{\rho}]$ is closed.  Any domain with an outgoing corner at some vertex of $\x[\SH_{\rho}]$ must cross one of the $z$ basepoints and is therefore excluded from the differential.  Statement (5) now follows immediately from Statement (4) and the definition of the map $G_{i,j}$ in Lemma~\ref{lemma:null-homotopy}.

Next, there is an obvious identification $f_s$ of the generators of $\SH_{\rho}$ and $\SH_s$ and since the stabilizations are performed in neighborhoods of $z$--basepoints, the complexes $\CFKt(\SH_{\rho})$ and $\CFKt(\SH_s)$ have identical differentials.  Thus $f_s$ is a chain map that sends $\x[\SH_{\rho}]$ to $\x[\SH_s]$.  The chain homotopy equivalences $f_h$ and $f$ are determined by counting triangles.  However, any triangle with vertices at $t_i$ and $\theta_i$ in Figure~\ref{fig:triangle} that misses the $z$ basepoints must also have a third corner at $t'_i$.  There is a unique such triangle, and it has a unique holomorphic representative.  This proves Statements (1) and (3).
\end{proof}

This suffices to show that $F_{0,0}( \x[\SH_{0}])$ is a scalar multiple of $\x[\SH_{0'}]$.  Thus, the transverse invariant is $0$ if $[\x[\SH_{0'}]] = 0$.  To finish the proof of (main theorem), we need to prove the converse as well.

\begin{proposition}
\label{prop:GF-transverse}
The chain-homotopy equivalence $F_{0,0}$ satisfies
\[F_{0,0}(\x[\SH_{0}]) = \x[\SH_{0'}]\]
\end{proposition}

\begin{proof}
By Lemma~\ref{lemma:transverse-subcomplex}, we have that
\[
	F_{0,0}( \x[\SH_{0}]) = \sum_{\substack{ \psi \in \pi_2(\Theta,\x[\SH_{0}],\x[\SH_{0'}]) \\ \mu(\psi) = 0 \\ n_z(\psi) = n_w(\psi) = 0}} \cN_{0,0}(\psi) \cdot \x[\SH_{0'}] 
\]
Thus, we need to determine which domains in $\pi_2(\Theta,\x[\SH_0],\x[\SH_{0'}])$ contribute to the triangle map.  The proof is similar to the proof of Lemma~\ref{lemma:bigon-arrow}.  Specifically, there is a unique positive domain $D(\psi_{small})$, its contribution to $F_{N,n_N}$ is exactly $1$, and its contribution to each successive map $F_{i,j}$ can never deviate from $1$.

First, since $\betas_{0'}$ consists of small Hamiltonian isotopes of $\betas_0$, there is a obvious small triangle $\psi_{small}$ in $\pi_2(\Theta,\x[\SH_0],\x[\SH_{0'}])$.  It appear as an embedded domain for every triple $\betas_i,\betas_0,\alpha$ for $i=0',1,\dots,N$.  If $\psi'$ is any other triangle in $\pi_2(\Theta,\x[\SH_0],\x[\SH_{0'}])$, then $D(\psi_{small} - \psi')$ is a multi-periodic domain.  If the domain $D(\psi')$ is positive, then its boundary must include some curve $\widehat{\alpha}_k$ with nonzero multiplicity.  But each $\widehat{\alpha}_k$ is linearly independent in $H_1(\Sigma_g)$ from the remaining $\alphas$ and $\betas$ curves.  This gives a contradiction so $D(\psi_{small})$ is the unique positive domain.

Secondly, in the triple $\betas_N,\betas_0,\alphas$, the domain appears and clearly has a unique holomorphic representative.  Thus $\cN_{N,n_N}(\psi_{small}) = 1$.  Moreover, at no point in the isotopy from $\SH_{0'}$ to $\SH_N$ does the triangle $\psi_{small}$ decompose into the sum of two positive domains $\psi$ and $\phi$.  This would require a finger move of some $\beta$ curve across the $\widehat{\alpha}_k$ boundary arc of $\psi_{small}$.  but such a finger move never happens in the isotopy.  Consequently, we must have that $\cN_{i',j'}(\psi_{small}) = 1$ for all $(i',j')$.
\end{proof}


\bibliographystyle{alpha}
\nocite{*}
\bibliography{References}

\end{document}